\documentclass[11pt]{article}
\usepackage[margin=1in]{geometry}
\usepackage[utf8]{inputenc}
\usepackage[T1]{fontenc}
\usepackage[english]{babel}
\usepackage{lmodern}
\usepackage{amsmath}
\usepackage{graphicx}
\usepackage{wrapfig}
\usepackage[subrefformat=parens,labelformat=parens]{subfig}
\usepackage{caption}
\usepackage{paralist}
\usepackage{enumitem, tikz, cleveref}

\newcommand{\titel}{
Corrigendum on Wiener index, Zagreb Indices and Harary index of Eulerian graphs}

\usepackage{amsthm}
\allowdisplaybreaks
\usepackage{amsfonts}
\theoremstyle{plain}
	\newtheorem{satz}{Satz}[]
	\newtheorem{theorem}[satz]{Theorem}

		\newtheorem{claim}[satz]{Claim}
\theoremstyle{remark}
	
\theoremstyle{definition}

\newcommand*{\myproofname}{Proof}
\newenvironment{claimproof}[1][\myproofname]{\begin{proof}[#1]}{\end{proof}}
\newtheorem*{claim*}{Claim}

\newcommand*{\ceilfrac}[2]{\mathopen{}\left\lceil\frac{#1}{#2}\right\rceil\mathclose{}}
\newcommand*{\floorfrac}[2]{\mathopen{}\left\lfloor\frac{#1}{#2}\right\rfloor\mathclose{}}
\newcommand*{\abs}[1]{\lvert #1\rvert}

\DeclareMathOperator{\ecc}{ecc}

\begin{document}
	\title{\titel}
		\author{
Stijn Cambie \thanks{Department of Computer Science, KU Leuven Campus Kulak-Kortrijk, 8500 Kortrijk, Belgium. Supported by Internal Funds of KU Leuven (PDM fellowship PDMT1/22/005).}
		}
	\date{}
	\maketitle

\begin{abstract}
    In the original article ``Wiener index of Eulerian graphs'' [Discrete Applied Mathematics
Volume 162, 10 January 2014, Pages 247-250], the authors state that the Wiener index (total distance) of an Eulerian graph is maximized by the cycle. We explain that the initial proof contains a flaw and note that it is a corollary of a result by Plesn{\'{\i}}k, since an Eulerian graph is $2$-edge-connected.
    The same incorrect proof is used in two referencing papers, 
    `Zagreb Indices and Multiplicative Zagreb Indices of Eulerian Graphs'' [Bull. Malays. Math. Sci. Soc. (2019) 42:67–78]
and ``Harary index of Eulerian graphs'' [J. Math. Chem., 59(5):1378–1394, 2021].
We give proofs of the main results of those papers and the $2$-edge-connected analogues.
\end{abstract}

\section*{Introduction}
The Wiener index or total distance $W(G)$ of a graph $G$ is the sum of all distances in a graph.
This graph parameter has been the inspiration for many interesting questions~\cite{KST23}. Together with the related average distance, it is the most fundamental distance-based topological index of a graph. 
An other distance-based topological index is the Harary index $H(G)$, which is the sum of reciprocals of the distances; $H(G)=\sum_{\{u,v\} \in \binom V2} \frac 1{d(u,v)}$.

The first Zagreb index, or the sum of squares of degrees, of a graph $G=(V,E)$ equals $M_1(G)=\sum_{v \in V} \deg^2(v).$
This is a rather natural degree-based topological index, and many more variants exist, e.g. $M_2(G)=\sum_{uv \in E} \deg(u)\deg(v), \Pi_1(G) = \prod_{v \in V} \deg(v) $ and $\Pi_2(G) = \prod_{v \in V} \deg(v)^{\deg(v)}.$

Some of the questions on the extremal behaviour for such topological indices are still open, while other questions turn out to boil down to known cases~\cite{SC21}. As such, the study should mainly focus on the most fundamental questions, that give the necessary insight for the related questions.

An Eulerian graph is a connected graph for which every vertex has even degree. This is equivalent to containing a closed walk (Eulerian cycle). The class of Eulerian graphs is, being related to the K\"onisberg Bridge problem and thus the origin of graph theory, one of the oldest graph classes.

In this article, we observe and fix mistakes in a few papers working on the above topological indices of Eulerian graphs, and generalize them to $2$-edge-connected graphs. 

\begin{theorem}
    Let $G$ be a Eulerian (or $2$-edge-connected) graph of order $n$,
    then $$W(G) \le W(C_n), H(G) \ge H(C_n) \mbox{ and } M_1(G) \ge M_1(C_n)$$
    with equality if and only if $G \cong C_n.$ 
\end{theorem}

Note that the bounds in the other direction are trivial to prove, and attained by $K_n$ ($n$ odd) or $K_n \setminus M$ (complement of a perfect matching, if $n$ is even).

\section*{Wiener index of Eulerian graphs}

The main result~\cite[Thr.~5]{GCR14} states that the Wiener index (total distance) of an Eulerian graph of order $n$ is maximized by the cycle $C_n$. 

\begin{theorem}\label{thr:W}
    For every Eulerian graph $G$ of order n, $W(G) \le W(C_n)$ with equality if and only if $G \cong C_n.$
\end{theorem}

Despite being cited over $40$ times, a flaw in the reasoning in its proof was not spotted before.
The initial idea proposed to prove this, is that deleting edges increases the Wiener index. By this observation, one can conclude that the Eulerian graph is edge-minimal, i.e. no strict subgraph of the extremal graph(s) can be spanning (connected and of order $n$) and Eulerian.
Since the cycle $C_n$ is not the only edge-minimal Eulerian graph of order $n$ (once $n \ge 5$), the conclusion cannot be drawn.

Here we note that Theorem~\ref{thr:W} is true, being a corollary of~\cite[Thr.~6]{Plensik84}.
An Eulerian graph $G$ of order $n$, having a closed walk, is $2$-edge-connected. Hence by~\cite[Thr.~6]{Plensik84} its Wiener index is bounded by $W(C_n)$, with equality if and only if $G=C_n$. Since $C_n$ is Eulerian, we conclude.
The latter was also observed by~\cite{Dankelmann21}, who also characterized the Eulerian graphs with second largest Wiener index.

\section*{Zagreb Indices and Multiplicative Zagreb Indices
of Eulerian Graphs}

The authors of~\cite{LWWW19} determined the Eulerian graphs with minimum Zagreb index using the same strategy as~\cite{GCR14}, which is again not correct (despite $>100$ citations), but the result~\cite[Thr.~6]{LWWW19} is.

\begin{theorem}
     For every Eulerian graph $G$ of order n, $$M_1(G) \ge M_1(C_n), M_2(G) \ge M_2(C_n), \Pi_1(G) \le \Pi_1(C_n) \mbox{ and } \Pi_2(G) \ge \Pi_2(C_n)$$ with equality if and only if $G \cong C_n.$
\end{theorem}
All inequalities are immediately true since an Eulerian graph has minimum degree $2$ and at least size $n,$ with equality if and only if it is a cycle.
As such $M_1(G) \ge n\cdot 2^2=4n$ and analogous.
Since the essence is that every quantity has to be minimal, also the inequality case is immediate.


If not all the degrees are equal to $2$, the degree sequence of the Eulerian graph majorizes $(4, \underbrace{2,2,\ldots,2}_{n-1 \text{ times}})$ and equality can only occur by two cycles with one vertex in common. As such, it is immediate that $M_1, \Pi_1$ and $\Pi_2$ attain the second-minimal value by these graphs.
Furthermore, the number of edges is at least $n+1$. 
As such, for the minimum of $M_2$, every edge contributes at least $4$, and at least $4$ edges contribute at least $8.$
This immediately implies that $M_2(G) \ge 4n+20$ with equality if and only if the Eulerian graph has $n+1$ edges.
A proof for $M_1$ and $M_2$ is also written down in detail in~\cite[Thr.~3.4]{LMC17} for the second-smallest case.


In~\cite[Thr.~3.3]{LMC17}, it is observed that one can consider a spanning unicyclic subgraph of an Eulerian graph to conclude the main case as a corollary of known results (without using the minimum degree being $2$).
For a unicyclic graph, the result is immediate by applying $QM-AM$ or the Cauchy-Schwarz inequality;
$\sum_v \deg^2(v) \ge \frac{ (\sum_v \deg(v) )^2}n=4n.$
Equality is true if and only if all degrees are equal to $2.$
For $M_2(G),$ one can start from a cycle and iteratively add one pendent edge between a new vertex $u$ and a vertex $v$ (belonging to the connected graph) at a time and note that the sum will increase with at least $4$ (at least $7$ in the first step), since the product on the added edge is at least $ 2$ and $\deg(v)$ increases with at least $1$, while it already had a neighbour with degree at least $2.$

\section*{Harary index of Eulerian graphs}

In~\cite[Thr.~6]{CWZ21}, an analogous incomplete proof has been stated. Deleting the largest number of edges does not imply the result, and the composition of $K_n$ into disjoint cycles does not imply that there is no other Eulerian subgraph, composed by other cycles, for which the Harary index is smaller.

As such, a proof that the Harary index is minimum for the cycle is lacking.
In this case, the authors observed that the case for the second-smallest was tricky.
Nevertheless, the given graph $G$ in~\cite[Fig.~1]{CWZ21} satisfies $H(G)=10\left(1+\frac 12\right)+4\left(\frac 13+ \frac14 \right)=\frac{52}{3}$ instead of $\frac{95}{6}$. As such, it is not the Eulerian graph of order $8$ with with second-minimal Harary index. The latter is attained by the composition of a $C_6$ and a $C_3$ with one vertex in common.

The fact that this case is harder, can be demonstrated by comparing the following bicyclic and tricyclic graphs for sufficiently large $n$;
\begin{itemize}
    \item $G_1$ being the composition of $C_{n-6}$ and $C_7$ with one vertex in common
    \item $G_2$ being a composition of a $C_{n-6}$ and $2$ $C_4$s
\end{itemize}
These two graphs are presented in~\cref{fig:graphs} (when $n=18)$.
Let $h$ be the composition of the two $2$ $C_4$s with a vertex in common.
Now one can compute that, for $n\ge 2k+1$ and $k \ge 10$, 
\begin{align*}
    H(G_1)-H(G_2)&\ge H(C_7)-H(h)+2\sum_{i=1}^{k -3} \left( \frac{1}{i+2}-\frac1{i+4} \right)\\
    &=-\frac 34+2\left(\frac 13 +\frac14 - \frac{1}{k}-\frac{1}{k+1}\right)>0.
\end{align*}
The latter since the distances between a vertex $u$ of the $C_{n-6}$ and a vertex $v_i$ of the $C_7$ resp. $h$ are equal, except for $v_3$. In that case, $d_{G_2}(u,v_3)=d_{G_1}(u,v_3)+2=d_{G_1}(u,v_1)+4.$
The tricyclic graph $G_2$ has thus smaller Harary index than the bicyclic $G_1$ when $n \ge 21$ (actually for $n \ge 20$).

\begin{figure}[ht]
\centering
\begin{tikzpicture}
\foreach \x in {0,30,...,330} {
    \draw[fill] (\x:2) circle (0.05);
    \draw (\x:2) -- ({\x+30}:2);
}

\foreach \i/\label in {0/v_1,1/v_2,2/v_3,3/v_4,4/v_5,5/v_6,6/v_7} {
    \draw[fill] ({-3.5+1.5*cos(\i*360/7)},{1.5*sin(\i*360/7)}) circle (0.05);
    \draw ({-3.5+1.5*cos(\i*360/7)},{1.5*sin(\i*360/7)}) -- ({-3.5+1.5*cos(\i*360/7+360/7)},{1.5*sin(\i*360/7+360/7)});
    \node at ({-3.5+1.8*cos(\i*360/7)},{1.8*sin(\i*360/7)}) {$\label$};
}
\end{tikzpicture}\quad
\begin{tikzpicture}
\foreach \x in {0,30,...,330} {
    \draw[fill] (\x:2) circle (0.05);
    \draw (\x:2) -- ({\x+30}:2);
}

\foreach \i/\label in {0/v_1,1/v_2,2/,3/v_7} {
    \draw[fill] ({-3+1*cos(\i*360/4)},{1*sin(\i*360/4)}) circle (0.05);
    \draw ({-3+1*cos(\i*360/4)},{1*sin(\i*360/4)}) -- ({-3+1*cos(\i*360/4+90)},{1*sin(\i*360/4+90)});
    \node at ({-3+1.3*cos(\i*360/4)},{1.3*sin(\i*360/4)}) {$\label$};
}

\foreach \i/\label in {0/v_6,1/v_4,2/,3/v_5} {
    \draw[fill] ({-5+1*cos(\i*360/4)},{1*sin(\i*360/4)}) circle (0.05);
    \draw ({-5+1*cos(\i*360/4)},{1*sin(\i*360/4)}) -- ({-5+1*cos(\i*360/4+90)},{1*sin(\i*360/4+90)});
    \node at ({-5+1.3*cos(\i*360/4)},{1.3*sin(\i*360/4)}) {$\label$};
}
\node at ({-5+0.7*cos(2*360/4)},{0.7*sin(2*360/4)}) {$v_3$};
\end{tikzpicture}

\caption{$G_1$ and $G_2$ for $n=18$}\label{fig:graphs}
\end{figure}
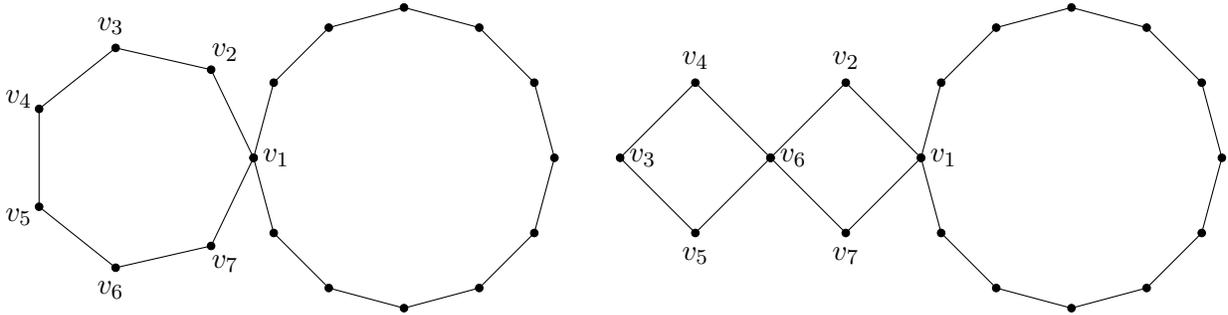

Finally, we prove the main theorem.
We first prove an analogue for $2$-connected graphs.

\begin{theorem}\label{thr:main_2conn}
    If $G=(V,E)$ is a $2$-connected graph on $n$ vertices, then $H(G) \ge H(C_n).$ Equality occurs if and only $G\cong C_n.$
\end{theorem}

\begin{proof}
    Take a vertex $v \in G.$
    Let the eccentricity of $v$ be $\ecc(v)=\max_{u \in V} d(u,v)$.
    Note that for every $1\le i \le \ecc(v)-1,$ there are at least $2$ vertices $u$ for which $d(u,v)=i$, as a single solution $u$ would be a cutvertex of $G$.
    This implies that the sequence $(d(u,v))_{u \in V \setminus v}$ is dominated by $\{1,1,2,2,3, \ldots, \floorfrac{n-1}2,\ceilfrac{n-1}{2}\}.$
    This implies that $\sum_{u \in V \setminus v} \frac{1}{d(u,v)} \ge \frac{2}{n} H(C_n).$ Summing over all $v \in V,$ we conclude that $H(G) \ge H(C_n).$
    If equality occurs, $\deg(v)=2$ $\forall v \in V$ and so $G \cong C_n.$ 
\end{proof}

Next, we prove the theorem for $2$-edge-connected graphs and thus also for Eulerian graphs.

\begin{theorem}\label{thr:main_2edgeconn}
    If $G=(V,E)$ is a $2$-edge-connected graph on $n$ vertices, then $H(G) \ge H(C_n).$ Equality occurs if and only $G\cong C_n.$
\end{theorem}

\begin{proof}
    If $G$ is $2$-connected, we conclude by Theorem~\ref{thr:main_2conn}. 
    Suppose by the contrary that $G$ is a minimum counterexample.
    Assume that $v$ is a cut-vertex of $G$, such that $v=G_1 \cap G_2$, $\abs {G_1} \le \abs {G_2}$ and $G_1$ is $2$-connected (we take a cut-vertex which is nearest to one of its endvertices).
    Let $V_1, V_2$ be the vertex sets of $G_1$ and $G_2$ without $v$.
    Then $G_1$ and $G_2$ are both $2$-edge-connected and thus $H(G_1) \ge H(C_{\abs{G_1}})$ and $H(G_2) \ge H(C_{\abs{G_2}})$ as $G$ was a minimum counterexample.
    Since the sequence $(d(u,v))_{u \in G_1}$ is dominated by the sequence when $G_1$ is a cycle, and $H(G_1) \ge H(C_{\abs{G_1}})$ is strict when $G_1$ is not a cycle, we can assume that $G_1$ is a cycle, i.e. $G_1=C_{\abs{G_1}}.$
    Since $G_2$ is ($2$-edge-)connected, the distances $(d(u,v))_{u \in G_2 \setminus v}$ are obviously dominated by $1,2,\ldots, \abs{G_2}-1,$ which would be the (impossible) case that $G_2$ is a path.
    Let $G'_2$ be the path $P_{\abs{G_2}}$, with $v$ one of its end-vertices.
    Let $G'=G'_2\cup G_1$ be this graph.
    Let $x=\floorfrac{\abs{G_1}}{2}$ and $y=\floorfrac{\abs{G_2}}{2}.$
    Here $y \ge x.$
    Consider the multiset $S$ containing all distances between vertices in $G_1=C_{\abs{G_1}}$, between vertices in $ C_{\abs{G_2}}$, and the distances $(d_{G'}(u',v'), v'\in V_1, u'\in V_2.$
    \begin{claim}\label{clm:1}
        The set $S$ contains every number from $1$ to $y-1$ at least $n$ times.  
        Every number which is at least $y+1$ appears at most $\abs{G_1}-1< \frac n2$ many times in $S$.
    \end{claim}
    \begin{claimproof}
        The distances $1$ to $x-1$ all appear $\abs{G_1}$ times in $G_1$. Furthermore $x$ appears $x$ times if $\abs{G_1}$ is even, and otherwise $\abs{G_1}$ many times.
        The distances $1$ to $y-1$ all appear $\abs{G_2}$ times in $C_{\abs{G_2}}.$
        If $x=y$, we are already done with the first part. So assume otherwise.
        If $\abs{G_1}=2x$ is even, every vertex except from $2$ ($v$ and the one diametrically opposite to $v$ in $G_1$) have at least one vertex in $G'_2$ which is at distance $x.$
        Since $\abs{G_1}-2=2x-2\ge x$ (note that $\abs{G_1}=2$ is impossible), $x$ appears at least $\abs{G_2}+2x>n$ many times among $s.$
        For every vertex $v' \in V_1$, every distance in $[x+1,y-1]$ appears exactly once as the distance towards a vertex of $G'_2.$ 

        Next, we consider distances above $y$.
        Remember that $G_1=C_{\abs{G_1}}, C_{\abs{G_2}}$ have diameter bounded by $y$ and since $G'_2$ is a path, for every $i>y$ and $v'\in V_1,$ there is at most one $u'\in V_2$ for which $d_{G'}(u',v')=i.$
    \end{claimproof}
    Let $T$ be the multiset of distances between vertices in $C_n$.
    Then trivially, $\abs T=\abs S.$
    If $n$ is odd, $T$ contains exactly the number from $1$ up to $\floorfrac n2>y$, each with multiplicity $n$.
    If $n=2k$ is even, $T$ contains $1$ up to $k-1$ with multiplicity $n$ and $k$ with multiplicity $k>\abs{G_1}-1.$
    This implies by Claim~\ref{clm:1} that $\max \{S \setminus T\} < \min \{ T \setminus S\}$ and thus
    $$\sum_{i \in S \setminus T} \frac 1i > \sum_{i \in T \setminus S} \frac 1i. $$
    From this, we conclude that 
    $W(C_n)=\sum_{i \in T} \frac 1i <\sum_{i \in S} \frac 1i \le W(G). $
\end{proof}

\section*{Conclusion}

We found a mistake in a proof that was used twice in further work, and fixed the proofs.
In each of these cases, the main result was true and the cycle $C_n$ is an extremal Eulerian graph for the studied topological indices. Alternative proofs could be derived from known cases by noting that an Eulerian graph is always $2$-edge-connected, and contain a unicyclic subgraph.
We want to raise awareness to authors that, especially in a field where many related questions can be posed, mathematicians should try to fully understand the underlying principles before publishing, to avoid creating an abundancy of papers that makes it harder for e.g. engineers to find the crucial insights.

\section*{Acknowledgement}
The author thanks Ivan Gutman and Jan Goedgebeur for some suggestions on the presentation of the corrigendum.

\bibliographystyle{abbrv}
\bibliography{ref}

\begin{thebibliography}{1}

\bibitem{CWZ21}
J.~Cai, P.~Wang, and L.~Zhang.
\newblock Harary index of {Eulerian} graphs.
\newblock {\em J. Math. Chem.}, 59(5):1378--1394, 2021.

\bibitem{SC21}
S.~Cambie.
\newblock Five results on maximizing topological indices in graphs.
\newblock {\em Discrete Math. Theor. Comput. Sci.}, 23(3):13, 2021.
\newblock Id/No 10.

\bibitem{Dankelmann21}
P.~Dankelmann.
\newblock Proof of a conjecture on the {Wiener} index of {Eulerian} graphs.
\newblock {\em Discrete Appl. Math.}, 301:99--108, 2021.

\bibitem{GCR14}
I.~Gutman, R.~Cruz, and J.~Rada.
\newblock Wiener index of {Eulerian} graphs.
\newblock {\em Discrete Appl. Math.}, 162:247--250, 2014.

\bibitem{KST23}
M.~{Knor}, R.~{{\v{S}}krekovski}, and A.~{Tepeh}.
\newblock {Selected topics on Wiener index}.
\newblock {\em arXiv e-prints}, page arXiv:2303.11405, Mar. 2023.

\bibitem{LWWW19}
J.-B. Liu, C.~Wang, S.~Wang, and B.~Wei.
\newblock Zagreb indices and multiplicative {Zagreb} indices of {Eulerian}
  graphs.
\newblock {\em Bull. Malays. Math. Sci. Soc. (2)}, 42(1):67--78, 2019.

\bibitem{LMC17}
Z.~Liu, Q.~Ma, and Y.~Chen.
\newblock New bounds on {Zagreb} indices.
\newblock {\em J. Math. Inequal.}, 11(1):167--179, 2017.

\bibitem{Plensik84}
J.~Plesn{\'{\i}}k.
\newblock On the sum of all distances in a graph or digraph.
\newblock {\em J. Graph Theory}, 8:1--21, 1984.

\end{thebibliography}

\end{document}